\documentclass[11pt, leqno]{article}

\usepackage{amsmath,amsthm}
\usepackage{amssymb}
\usepackage{amsbsy}
\usepackage{enumerate}
\usepackage{indentfirst}
\usepackage{latexsym}
\usepackage{epic}
\usepackage{amsfonts}
\usepackage{color}
\usepackage{dsfont}
\usepackage{todonotes}

\usepackage{enumitem}

\usepackage{graphicx}

\usepackage[T1]{fontenc}

\newtheorem{theorem}{Theorem}[section]
\newtheorem{corollary}[theorem]{Corollary}
\newtheorem{lemma}[theorem]{Lemma}
\newtheorem{proposition}[theorem]{Proposition}
\newtheorem{observation}[theorem]{Observation}

\theoremstyle{definition}

\newtheorem{example}[theorem]{Example}

\newtheorem*{mainthm}{Main Theorem}

\numberwithin{equation}{section}

\newcommand{\free}{\mathrm{F}}
\newcommand{\wor}{\mathrm{W}}

\begin{document}


\baselineskip=17pt


\title{On groups with weak Sierpi\'nski subsets}

\author{Agnieszka Bier,
Yves de Cornulier and Piotr S{\l}anina}

\date{March 17, 2019}

\maketitle







\begin{abstract}
In a group $G$, a weak Sierpi\'nski subset is a subset $E$ such that for some $g,h\in G$ and $a\neq b\in E$, we have $gE=E\smallsetminus \{a\}$ and $hE=E\smallsetminus \{b\}$. In this setting, we study the subgroup generated by $g$ and $h$, and show that has a special presentation, namely of the form $G_k=\langle g,h\mid (h^{-1}g)^k\rangle$ unless it is free over $(g,h)$. In addition, in such groups $G_k$, we characterize all weak Sierpi\'nski subsets.
\end{abstract}

\section{Introduction}
For a group $G$ and $g, h\in G$ we say that a subset $E\subseteq G$ is a $(g,h)$-wS-subset (a \emph{weak Sierpi\'nski} subset with respect to $g$ and $h$), if there exist elements $a\neq  b \in E$ such that $gE=E\smallsetminus \{a\}$ and $hE=E\smallsetminus \{b\}$. We will also use the tuple notation $(E,g,h,a,b)$ whenever the elements $a$ and $b$, called removable points, should be indicated directly. A subset $E$ is called a $wS$-subset, if it is a $(g,h)$-wS-subset for some $g,h\in G$.

Our first result is that the subgroup generated by $(g,h)$ has a very special form:

\begin{mainthm}
Let $G$ be a group with a $(g,h)$-wS-subset. Then the subgroup $H=\langle g,h\rangle$ is either free over $(g,h)$, or there exists $k\ge 2$ such that it has the presentation $H=G_k=\langle g,h\mid (h^{-1}g)^k\rangle$.
\end{mainthm}

The group $G_k$ is the free product of its infinite cyclic subgroup $\langle g\rangle$ and its cyclic subgroup of order $k$ $\langle h^{-1}g\rangle$, and therefore has a well-understood structure. For instance, by Kurosh's theorem \cite{kurosh} (see also \cite{LS}), its torsion elements are conjugated to elements of the finite cyclic subgroup $\langle h^{-1}g\rangle$.
It also follows, in all cases, that $(g,h^{-1}gh)$ is a free family in $H$, and in turn this yields:

\begin{corollary}\label{cor}
Let $G$ be a group. If $G$ contains a wS-subset then $G$ contains a free subgroup of rank two.
\end{corollary}

The origins of the problem of existence of wS-subsets in groups go back to works of Sierpi\'nski, who considered subsets with so-called removable points in Euclidean $n$-dimensional space $\mathbb{R}^n$ \cite{Sier}.
In the following studies of Sierpi\'nski \cite{Sier}, \cite{Siermon}, Mycielski \cite{Myc} and Straus \cite{Straus}, \cite{Straus2} three kinds of subsets of $\mathbb{R}^n$ appeared to be of special interest: subsets with at most one removable point, subsets with two distinct removable points (weak Sierpi\'nski subsets) and subsets in which  all points are removable (Sierpi\'nski subsets). The ideas developed in this context were transferred by Straus to subsets of a group acting on itself  by left translations. In \cite{Straus} and \cite{Straus2} he proved that a group contains a Sierpi\'nski subset if and only if it contains a free subgroup of rank two.
Whether the same can be said about groups with weak Sierpi\'nski subsets became then an interesting question and as such it  was posed by  Mycielski and Tomkowicz in \cite{MycTom} and  \cite{TomWagbook} (Question 7.22).
Corollary \ref{cor}  gives the affirmative answer to this question. We remark however, that this statement  can directly be deduced from the Stallings splitting theorem (see Section \ref{shorthm} for details)

The paper is organized in four sections.
 Section 2 contains some basic definitions and preliminary observations. In Section 3 we study relations in a non-free group $G$  containing a weak Sierpi\'nski subset. We prove the main theorem and, as  corollaries, we find a complete description of wS-subsets in the respective subgroups $G_k$ and characterize the non-free subgroups of $G_k$.
In the last section we provide a shorter proof of Corollary \ref{cor} using the concept of ends of groups.

\section{Notation and preliminaries}

We begin with some simple observations on properties of groups with wS-subsets.

In what follows, $\wor(x,y)$ denotes the free monoid on the four symbols $x^{\pm 1}$, $y^{\pm 1}$; its elements are called words, and typically represent elements of the free group $\free_2(x,y)$ on $(x,y)$. The length $|w|$ of a word $w$ is the number of its constituting letters, for instance the length of $xyy^{-1}$ is 3.

For a group $G$ with a pair $(g,h)$, we consider the (oriented, labeled) Schreier graph $\Gamma(G,g,h)$ consisting of the set of vertices $G$, with, for each $x\in G$, an edge from $x$ to $gx$ labeled $g$ and an edge from $x$ to $hx$ labeled $h$. Note that the connected component of $x\in G$ in this graph is the right $H$-cosets $Hx$. In particular, the Schreier graph $\Gamma(H,g,h)$ is a left Cayley graph of $H$.

If $p$ is a path in $\Gamma (G,g,h)$, then by $|p|$ we denote the length of $p$, i.e. the number of constituent edges of $p$. We note that if $w$ is the word obtained by concatenating the labels of consecutive edges in $p$, then $|p| = |w|$, so we use one symbol $|\cdot |$ for both notions, the length of the path and length of the word.

\begin{observation}\label{ab}
Let $G$ be an abelian group. Then $G$ contains no wS-subset.
\end{observation}
\noindent\textit{Proof. } Let $E$ be a wS-subset.
 Then by definition:
 $$E\smallsetminus\{b, ha\}=h(E\smallsetminus\{a\})=(hg)E = (gh)E = g(E\smallsetminus\{b\}) =  E\smallsetminus\{a, gb\}$$
and hence either $a=b$ or $a=ha$ and $h=1$. We get a contradiction. $\Box$
 \smallskip

\begin{observation}
Let $E$ be a $(g,h)$-wS-subset of $G$. Then, for every $\gamma$ in the (non-unital) subsemigroup generated by $g,h$, $\gamma E$ is properly contained in $E$. In particular, $g,h$ are not torsion.
\end{observation}

 \noindent\textit{Proof. }
Since $gE$ and $hE$ are properly contained in $E$, the first assertion is follows by induction. The second follows, applied when $\gamma$ is a positive power of $g$ or $h$. Assume $g^k = 1$ for some $k\in \mathbb{N}$. Then $g^kE=1\cdot E = E$, and that contradicts the first statement.
$\Box$
\smallskip

\section{Non-free groups with wS-subsets}

In this section we investigate subgroups  of  groups containing a $(g,h)$-wS-subset, generated by elements $g$ and $h$. As shown in the following example, the subgroup $\langle g,h\rangle$ need not be free.

\begin{example}\label{examp}Suppose $g$ is an element of infinite order and $s$ is an element of order $k\ge 2$.  Let $G=\langle g\rangle\ast\langle s\rangle$ be the free product and write $h=sg$. Define $E\subseteq G$ as the set of all elements represented by reduced words in  $g$ and $s$  ending with a power of $g$. Then
$gE=E\smallsetminus\{g\}$ and $hE=sgE=E\smallsetminus\{sg\}$ and $g^{-1}h$ has finite order.
\end{example}
\smallskip

We note that this fact was independently observed by Tomkowicz (private communication). We continue our investigations  with the assumption that $(g,h)$ is not free (by this we mean that the the subgroup generated by $g,h$ is not free over the pair $(g,h)$).

\begin{lemma}\label{Hhastoo}
If $G$ contains a $(g,h)$-wS-subset and $(g,h)$ is not free, then the subgroup $H=\langle g,h\rangle$ also has a $(g,h)$-wS-subset.
\end{lemma}

\noindent\textit{Proof. }
Let $E$ be a $(g,h)$-wS-subset. For at most one right $H$-coset $Hu$, the intersection  $Hu\cap E$ is
not left $g$-invariant, and at most one, say $Hv$, the intersection $Hv\cap
E$ is not $h$-invariant. We may assume that $u=1$ (right translating by
$u^{-1}$). Consider the Schreier graph $\Gamma(G,g,h)$. For subset $G'\subset G$, the boundary of subgraph $\Gamma'(G',g,h)$ is the set of all edges in $\Gamma(G,g,h)$ connecting vertices in $G'$ with vertices in $G\setminus G'$ or vice versa. If  $v\notin H$, then $E\cap H$ is a subset whose boundary is a
single edge labeled by $g$. This edge cannot hence be contained in any cycle. The Schreier graph $\Gamma(G,g,h)$ is a transitive graph under the right action of $G$.
It follows that no edge labeled by $g$ belongs to a loop. Similarly (working in the other coset) edges labeled by $h$
cannot be contained in any loop. Hence the component $H$ of the Schreier graph is a tree (with our definition of Schreier graph, which allowed a priori double edges or self-loops). This means that $H$ is freely generated by $(g,h)$. $\Box$

\bigskip

\noindent\textbf{Proof of Main Theorem. }
We assume that $H$ is not freely generated by $(g,h)$, and hence by Lemma \ref{Hhastoo}, we can suppose that $E\subseteq H$.

In the left Cayley graph $\Gamma(H,g,h)$, call critical edges the two boundary edges $c_g$
and $c_h$ of $E$ (one is labeled by $g$ and one by $h$). Consider critical
loops, that is, geodesic loops passing through one (and hence
both) critical edge (see Figure 1).

\begin{figure}\label{fi1}

\setlength{\unitlength}{0.6mm}
\begin{picture}(200,100)(-100,0)
\put(-50,0){\line(0,1){80}}
\put(-50,0){\line(1,0){100}}
\put(50,80){\line(0,-1){80}}
\put(50,80){\line(-1,0){100}}
\put(-45,70){$E$}

\put(70,15){\vector(-1,0){40}}
\put(70,65){\vector(-1,0){40}}
\put(53,67){$c_h$}
\put(53,18){$c_g$}
\put(79,40){$\beta_0$}
\put(24,40){$\alpha_0$}

\put(70,15){\circle*{0.5}}
\put(30,15){\circle*{0.5}}
\put(70,65){\circle*{0.5}}
\put(30,65){\circle*{0.5}}
\put(68,11){${g^{-1}a}$}
\put(27,11){${a}$}
\put(70,66){${h^{-1}b}$}
\put(30,66){${b}$}

\linethickness{0.05mm}
	\qbezier(70, 15)(85, 40)(70, 65)
	\qbezier(30, 15)(15, 40)(30, 65)
	\put(77.5,42){\vector(0,-1){4}}
	\put(22.5,38){\vector(0,1){4}}
		
\end{picture}
\caption{A critical loop (see the proof of the Main Theorem)
}
\end{figure}
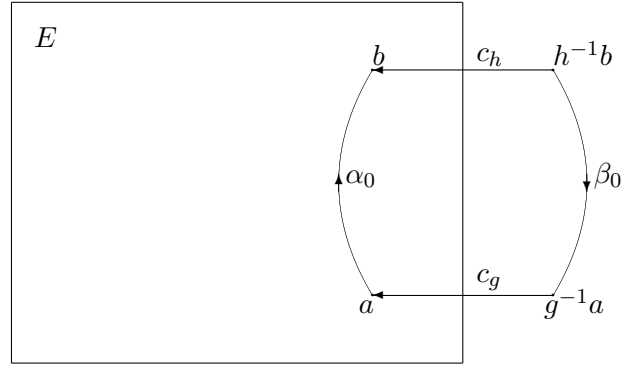

Let us start by the observation that $(E,g,h,a,b)$ play the almost same role as $(E^c,g^{-1},h^{-1},g^{-1}a,h^{-1}b)$, so $E$ and its complement $E^c$ play almost symmetric roles; the subsets $E$ and $E^c$ are joined exactly by two edges $(g^{-1}a,a)$ and $(h^{-1}b,b)$, labeled by $g$ and $h$. The ``almost" is due to the fact that one might have $g^{-1}a=h^{-1}b$, which has to be taken care of.

Let $n_0\ge 1$ be the smallest size of a nontrivial loop in $H$ (it exists by non-freeness). Translating, we can suppose that this loop is critical: write it as
$u_0=\beta_0\bar{c}_h\alpha_0c_g$ (see Figure 1).

We can suppose that $|\alpha_0|\ge |\beta_0|$. Indeed, otherwise, $|\beta_0|>|\alpha_0|$ and in particular $|\beta_0|\ge 1$; being geodesic implies $g^{-1}a\neq h^{-1}b$, and in this case we can replace $E$ with $E^{c}$ as indicated above.

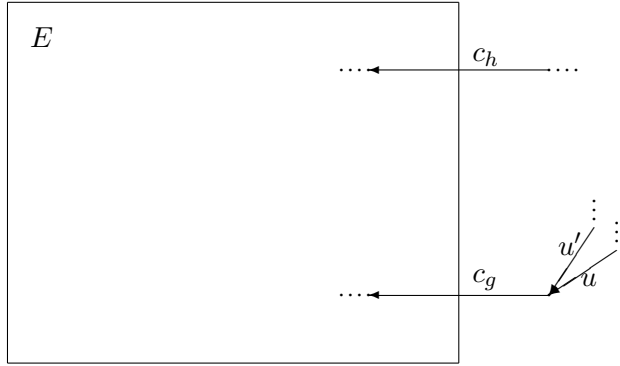
\begin{figure}\label{fi2}

\setlength{\unitlength}{0.6mm}
\begin{picture}(200,100)(-100,0)
\put(-50,0){\line(0,1){80}}
\put(-50,0){\line(1,0){100}}
\put(50,80){\line(0,-1){80}}
\put(50,80){\line(-1,0){100}}
\put(-45,70){$E$}

\put(70,15){\vector(-1,0){40}}
\put(70,65){\vector(-1,0){40}}
\put(53,67){$c_h$}
\put(53,18){$c_g$}
\put(85,25){\vector(-3,-2){15}}
\put(77,17){$u$}
\put(80,30){\vector(-2,-3){10}}
\put(72,24){$u'$}

\put(70,15){\circle*{0.5}}
\put(30,15){\circle*{0.5}}
\put(70,65){\circle*{0.5}}
\put(30,65){\circle*{0.5}}

\linethickness{0.05mm}

\put(28,15){\circle*{0.2}}
\put(26,15){\circle*{0.2}}
\put(24,15){\circle*{0.2}}

\put(85,27){\circle*{0.2}}
\put(85,29){\circle*{0.2}}
\put(85,31){\circle*{0.2}}

\put(80,32){\circle*{0.2}}
\put(80,34){\circle*{0.2}}
\put(80,36){\circle*{0.2}}
		
\put(72,65){\circle*{0.2}}
\put(74,65){\circle*{0.2}}
\put(76,65){\circle*{0.2}}

\put(28,65){\circle*{0.2}}
\put(26,65){\circle*{0.2}}
\put(24,65){\circle*{0.2}}

\end{picture}
\caption{Critical edges and forking.}
\end{figure}

This reduction being made, we have the following:

\begin{lemma}\label{critical_b0}
Every critical loop meets $E^c$ in exactly the path $\beta_0$.
\end{lemma}
\begin{proof}
Consider a critical loop
$u=\beta\bar{c}_h\alpha c_g$, say of length $n\ge n_0$; we need to prove $\beta=\beta_0$.

If $\beta$ is geodesic, then $|\beta|=|\beta_0|$, so the loop $\bar{\beta}\beta_0$ has length $<n_0$, and we deduce $\beta=\beta_0$.

Suppose now that $\beta$ is not geodesic, so that $|\beta|>n/2$. So $\bar{c}_h\alpha c_g$ (and hence $\alpha$) is geodesic, and in particular $|\alpha|\le |\alpha_0|$. If $|\alpha|<|\alpha_0|$ then the loop $\beta_0\bar{c}_h\alpha c_g$ has length $<n_0$ and this contradicts the minimality of $n_0$, so $|\alpha|=|\alpha_0|$. This implies that $\alpha_0$ is geodesic. In
turn, $\bar{c}_h\alpha_0c_g$ is geodesic as well (otherwise we would
contradict that $\bar{c}_h\alpha c_g$ is geodesic). This is a contradiction, since this is a geodesic path of length $\ge n_0/2+1$ in a loop of size $n_0$ (whose largest geodesic path has length $\le n_0/2$).
\end{proof}

\begin{lemma}\label{word_lemma}
Say that a cyclically reduced word $w'$ is bad if it is not (up to cyclic conjugation or inversion of the form) $g^n$, $h^n$, $(gh)^n$, or $(h^{-1}g)^n$ for any $n\ge 0$. For every bad $w'$ either there exists two directed edges labeled $g$ such that the previous labels (the label of the edge preceding the given edges, whose labels have to be $g$, $h$ or $h^{-1}$) differ, or there exists two directed edges labeled $h$ such that the previous labels ($h$, $g$ or $g^{-1}$) differ.
\end{lemma}
\begin{proof}
Details are omitted. If there exists no two directed edges labeled $g$ such that the next labels differ, one readily sees that $w'$ has the form (up to inversion and cyclic conjugation) $g^n$, $h^n$, or $gh^{n_1}\dots gh^{n_k}$ for some $k\ge 1$ and $n_1,\dots,n_k$, all positive or all negative. If furthermore the symmetric condition holds for $h$, we deduce that $w'$ is not bad (i.e., in the third case, one obtains that all $n_i$ are equal to $1$ or $-1$.
\end{proof}

Now consider a critical loop $u$. If it is bad, by Lemma \ref{word_lemma} it has a right translate $u'$ which forks before $c_g$ or $c_h$ (see Figure 2). Then $u'$ is also critical and we contradict Lemma \ref{critical_b0}. (In the case $g^{-1}a=h^{-1}b$ this is also a contradiction since then the loop cannot return into $E$.)

Hence every critical loop is not bad. Since every critical loop, up to inversion, both contains $g$ and $h^{-1}$, the only possibility is that the boundary label is $(h^{-1}g)^n$ for some $n\ge 1$. Since by assumption $(h^{-1}g)^{n_0}=1$ and is the smallest loop, this only holds for $n=n_0$.

Since every geodesic loop has a critical
right-translate, we conclude that there is a unique (up to right-translation
and choice of orientation) geodesic loop of positive length in the left Cayley
graph, namely $u_0$.

It follows that $G=\langle g,h\mid (h^{-1}g)^{n_0}\rangle$.  $\Box$
\smallskip

In addition, we can fully describe the  $wS$-subsets in the group
$G_k=\langle g,h\mid (h^{-1}g)^k=1\rangle$, $k\ge 2$:

\begin{proposition}\label{descri}
In $G_k$, there are exactly $k$ subsets $E$ such that
$gE=E\smallsetminus\{1\}$ and $hE$ is $E$ minus a singleton; for $k-1$ of
them this yields a wS-subset. More precisely, in the Schreier graph, these
are the subsets $E_\ell$ defined by cutting along the edge $(g^{-1},1)$
and the edge $(g^{-1}(hg^{-1})^{\ell-1} ,(hg^{-1})^\ell)$ for some
$1\le\ell\le k$. We have $hE_\ell=E_\ell\smallsetminus \{b_\ell\}$ with
$b_\ell=(hg^{-1})^{\ell}$, which for $\ell=k$ equals 1 and otherwise is
not 1 (so we have a wS-subset).

In particular the right action of $G_k$ on the set of $(g,h)$-wS-subsets is free and has exactly $k-1$ orbits.
\end{proposition}
\smallskip

\noindent\textit{Proof. }
This is clear by drawing the left Cayley graph of $G_k$, which is a ``tree-like" union
of $2k$-gons (see Figure 3). After cutting the edge $(g^{-1},1)$, the only other edges
that cut into 2 components after removal, are the $2k-1$ remaining ones in the same hexagon, $k$ of which being labeled by $h$, and are as described in the statement.

In the last assertion, the freeness follows since each $(g,h)$-wS-subset $E$ has a unique boundary edge labeled by $g$ and the action on the set of edges is free. The statement about the number of orbits follows from the first assertion, since each wS-subset $(g,h)$-wS-subset $E$ has a unique right translate $E'$ such that $E'\smallsetminus gE'=\{1\}$. $\Box$
\smallskip


\begin{figure}\label{fi3}
\centering\includegraphics{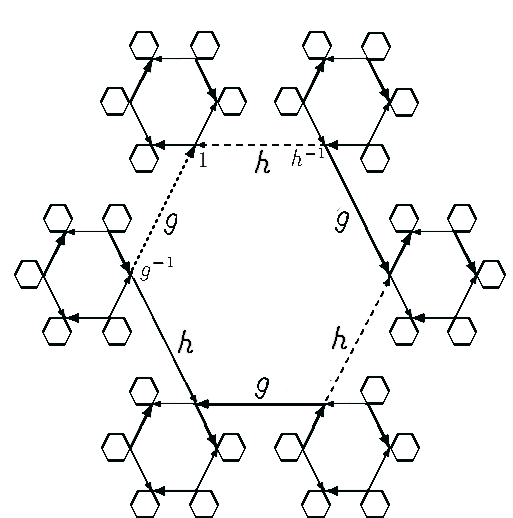}
\caption{The Cayley graph of $G_3$ -- a net of hexagons}
\end{figure}

\section{Proof of Corollary \ref{cor} using Stallings' theorem}\label{shorthm}

We now give the proof of Corollary \ref{cor}, using the concept of ends of groups. In fact we prove \ref{cor} in a more general setting, allowing singletons $\{a\},\{b\}$ to be replaced by disjoint nonempty finite subsets. We use again Lemma \ref{Hhastoo}, allowing to assume that $G$ is generated by $\{g,h\}$.

A group $G$ is said to have $\ge 2$ ends if there exists a subset $E$ of $G$ that is $G$-commensurated, in the sense that $gE\bigtriangleup E$  is finite for every $g\in E$, and
such that both $E$ and $G\smallsetminus E$ are infinite.

\begin{proposition}
Consider a group $G$ with two elements $g,h$, a subset $E$ and disjoint nonempty finite subsets $A,B$ such that $gE=E\smallsetminus A$ and $hE=E\smallsetminus B$. Then $G$ contains a nonabelian free subgroup.
\end{proposition}
\begin{proof}
Denote by $H$ the subgroup generated by $\{g,h\}$. Multiplying on the right, we can suppose that $1\in A$. Then $E'=E\cap H$ is $H$-commensurated; moreover, it is both infinite and has infinite complement in $H$: indeed, for $A'=A\cap H=E'\smallsetminus gA'$, the $g^nA'$ for $n\ge 0$ are pairwise disjoint, nonempty subsets of $E$, and the $g^{-n}A'$, for $n\ge 1$ are pairwise disjoint, nonempty subsets of $H\smallsetminus E$. Hence, $H$ has at least 2 ends.

If $H$ has $\ge 3$ ends, Stallings' theorem \cite{Stall} implies that $H$ has a free subgroup.

It remains to suppose that $H$ has 2 ends. For $\gamma\in G$, write $f(\gamma)=|E\smallsetminus\gamma E|-|\gamma E\smallsetminus E|$. By an elementary verification, $f$ is a group homomorphism. Then $f(g)>0$, so $f$ is nonzero and hence, since $H$ is 2-ended, its kernel is finite. Then, up to finite symmetric difference, the only $H$-commensurated subsets of $H$ that are infinite with infinite complement are $f^{-1}(\mathbf{N})$ and its complement. Since $f(E')\subset E'$, it follows that $E'$ has finite symmetric difference with $f^{-1}(\mathbb{N})$. Let $x$ be an element of $E'$ with $f(x)$ minimal (it thus exists). Since $f(g)>0$ and $f(h)>0$, we see that $x\notin gE'$ and $x\notin hE'$. So $x\in A\cap B$, a contradiction.
\end{proof}

It would be interesting to produce an explicit free subgroup in this setting. Also, we do not know if we can reach the same conclusion when $A$ or $B$ is infinite.

\subsection*{Acknowledgements}
We would like to thank the referee for useful reading, which led us to improve the exposition of the proof of the Main Theorem.

We learned about wS-subsets and related problems from Grzegorz Tom\-ko\-wicz, when he gave a talk at the algebraic seminar in Gliwice in November 2017. We are grateful to him for this introduction, some comments concerning our preliminary results and for sharing the text of \cite{MycTom}.
After this work has been accepted for publication, Grzegorz informed us that he and Mycielski obtained the main theorem independently in \cite{MT_graph}.

Agnieszka and Piotr thank Aleksander Ivanov for inspiring discussions and valuable comments on the course of their work.

Y.C. was supported by ANR GAMME (ANR-14-CE25-0004).

\end{document}